\documentclass[11pt]{article}
\usepackage{amsfonts}
\usepackage{bookmark}
\usepackage{color}

\usepackage{amsmath,amsfonts,amssymb,amscd,amsthm,amsbsy,bm,latexsym}
\usepackage [latin1]{inputenc}

\newtheorem{lemma}{Lemma}[section]
\newtheorem{theorem}{Theorem}[section]
\newtheorem{definition}{Definition}[section]
\newtheorem{proposition}{Proposition}[section]

\numberwithin{equation}{section}

\begin{document}

\title{ Long-time dynamical behavior for a piezoelectric system with magnetic effect and nonlinear dampings \footnote{ Email addresses:  gongweiliu@haut.edu.cn (G. Liu), wangmengru97@163.com (M. Wang), dpymath@haut.edu.cn (P. Ding)}}

\author
{Gongwei Liu,\ Mengru Wang,\ Pengyan Ding\\
College of Science, Henan University of Technology, Zhengzhou 450001, China}

\date{}
\vskip 0.2cm
\maketitle
\noindent{\bf Abstract.}  This paper is concerned with the long-time dynamical behavior of a piezoelectric system with magnetic effect, which has nonlinear damping terms and external forces with a parameter. At first, we use the nonlinear semigroup theory to prove the well-posedness of solutions. Then, we investigate the properties of global attractors and the existence of exponential attractors. Finally, the upper semicontinuity of global attractors has been investigated.



\section{Introduction}
\setcounter{equation}{0}
In this paper, we investigate the piezoelectric beam system with magnetic effect:
\begin{equation}\label{1.1}
\begin{cases}
\ \rho v_{tt}-\alpha v_{xx}+\gamma\beta p_{xx}+f_{1}(v,p)+g_{1}(v_{t})=\varepsilon h_{1},~~~\quad(0,L)\times(0,T),\\[2mm]
\mu p_{tt}-\beta p_{xx}+\gamma\beta v_{xx}+f_{2}(v,p)+g_{2}(p_{t})=\varepsilon h_{2},~~~\quad(0,L)\times(0,T),\\[2mm]
\end{cases}
\end{equation}
where $\varepsilon \in [0,1], \alpha=\alpha_{1}+\gamma^{2}\beta$, and $\alpha,\rho,\gamma,\beta,\mu>0$ mean elastic stiffness, the mass density per unit volume, piezoelectric coefficient, water resistance coefficient, magnetic permeability, respectively,
which is supplemented by the following initial boundary conditions:
\begin{equation}\label{1.2}
\begin{cases}
v(0,t)=\alpha v_{x}(L,t)-\gamma \beta p_{x}(L,t)=0,~~~\quad t>0,\\[2mm]
p(0,t)=p_{x}(L,t)-\gamma v_{x}(L,t)=0,~~~~~~~~\quad t>0,\\[2mm]
v(x,0)=v_{0},\quad v_{t}(x,0)=v_{1},~~~~~~~~~~\quad 0<x<L,\\[2mm]
p(x,0)=p_{0},\quad p_{t}(x,0)=p_{1},~~~~~~~~~~\quad 0<x<L.\\[2mm]
\end{cases}
\end{equation}

It is well known that the ability to convert mechanical energy and electrical energy into each other is a characteristic of piezoelectric materials. In addition, it is  smaller, less expensive, and more efficient \cite{z2021}. Therefore, it has been applied in automotive \cite{Li2015, Xu2019}, medical \cite{Da2016, Dag2014}, space structures \cite{Bu2014, Ma2015} and many other fields \cite{Zh2015}. In the previous modeling stage, the magnetic effect was basically ignored due to its slight \cite{Ra2021}, but the latest results show that the magnetic effect may influence the performance of the system \cite{Sm2005}, so it is crucial to study piezoelectric beam models with magnetic effects. Nowadays, a growing number of scholars are beginning to do related research in this field.

The establishment of mathematical models of piezoelectric beams has a long history. In the early days, Tebou \cite{Te2006} and Haraux \cite{Ha1989} established the model of a single beam due to Maxwell equation \cite{Ma1982,Ma1998} and the dynamic interaction between electromagnetism is ignored. Therefore, they obtained  the following equation:
\begin{equation*}
\begin{cases}
\rho v_{tt}-\alpha_{1}v_{xx}=0,~~~~~~~~~~~~~~\qquad \qquad(x,t)\in(0,L)\times\mathbb{R}^+,\\[2mm]
 v(0,t)=\alpha_{1}v_{x}(L,t)+\delta v_{t}(L,t)=0,~\quad t\in \mathbb{R}^+.\\[2mm]
\end{cases}
\end{equation*}
Based on the magnetic effects and variational method, the equations of a single beam were derived by Morris and \"{O}zer \cite{Mo2013}, as follows:
\begin{equation}\label{1.3}
\begin{cases}
\rho v_{tt}-\alpha v_{xx}+\gamma\beta p_{xx}=0,~~\qquad \qquad \qquad \qquad \qquad (x,t)\in(0,L)\times\mathbb{R}^+,\\[2mm]
\mu p_{tt}-\beta p_{xx}+\gamma\beta v_{xx}=0,~~~\qquad \qquad \qquad \qquad \qquad (x,t)\in(0,L)\times\mathbb{R}^+,\\[2mm]
v(0,t)=p(0,t)=\alpha v_{x}(L,t)-\gamma\beta p_{x}(L,t)=0,~~~~\quad t\in\mathbb{R}^+,\\[2mm]
\beta p_{x}(L,t)+\frac{V(t)}{h}=0,~~~~~~~~~\qquad \qquad \qquad \qquad \qquad t\in\mathbb{R}^+.\\[2mm]
\end{cases}
\end{equation}
where $V(t)=p_{t}(L,t)$ denotes electrical feedback controller. Then, they established  the further results: the well-posedness of solution, strong stability of piezoelectric beam and so on \cite{Mo2014,z2015}.\\
\hspace*{0.7cm}In \cite{Ra2018}, Ramos et al. inserted a dissipative term $\delta v_{t}$ in the first equation of \eqref{1.3} and considered the following boundary condition
\begin{equation*}
\begin{cases}
v(0,t)=\alpha v_{x}(L,t)-\gamma\beta p_{x}(L,t)=0,~~~~\quad 0<t<T,\\[2mm]
p(0,t)=p_{x}(L,t)-\gamma v_{x}(L,t)=0,~~~~~~~~\quad 0<t<T.\\[2mm]
\end{cases}
\end{equation*}
The authors applied energy method to prove the energy of the system and estimated that the energy is exponentially stable. After that, Ramos et al. changed the boundary condition in \cite{Ra2019}:
\begin{equation*}
\begin{cases}
v(0,t)=\alpha v_{x}(L,t)-\gamma\beta p_{x}(L,t)+\xi_{1}\frac{v_{t}(L,t)}{h}=0,~~~~\quad 0<t<T,\\[2mm]
p(0,t)=p_{x}(L,t)-\gamma v_{x}(L,t)+\xi_{2}\frac{p_{t}(L,t)}{h}=0,~~~~~~~~\quad 0<t<T,\\[2mm]
\end{cases}
\end{equation*}
and proved the exponential stability regardless of any relationship between system coefficients and there is equivalent to exact observability at the boundary.

Moreover, we know there are many ways to control piezoelectric vibration, among which time-delayed feedback control is a common way to improve stability of the system. Therefore, more and more researchers studied the influence of time-delayed effect on the stability of hyperbolic systems, and also proved the so-called destabilizing effect.

Dakto et al. \cite{Da1978,Da1986,Da1988,Da1991} has obtained a series of research results in this direction. Among them, Dakto studied the equation as follows,
\begin{equation*}
u_{tt}-u_{xx}+2u_{t}(x,t-\tau)=0,~~\quad (x,t)\in(0,1)\times(0,+\infty),
\end{equation*}
where $\tau >0$, and proved that the time delay in the dissipative term charcatered by velocity term could  make the beams lose stability.\\
\hspace*{0.7cm}Therefore, it is necessary to add some controlled term to stabilize the hyperbolic systems.

In 2021, Freitas and Ramos \cite{Fr2021} considered the longitudinal vibration of the piezoelectric beams with thermal effect, magnetic effect and fraction damping:
\begin{equation*}
\begin{cases}
\rho v_{tt}-\alpha v_{xx}+\gamma\beta p_{xx}+\delta\theta_{x}+f_{1}(v,p)+v_{t}=h_{1},~~~\quad (0,L)\times(0,T),\\[2mm]
\mu p_{tt}-\beta p_{xx}+\gamma\beta v_{xx}+A^{\nu}p_{t}+f_{2}(v,p)=h_{2},~~~~~~~\quad (0,L)\times(0,T),\\[2mm]
c\theta_{t}-\kappa\theta_{xx}+\delta v_{tx}=0,~~~~~~~~~\qquad \qquad \quad \quad \qquad \qquad(0,L)\times(0,T),\\[2mm]
\end{cases}
\end{equation*}
the authors applied the semigroup theory to prove the well-posedness of solutions, then they showed that the existence of global attractor, exponential attractor and the upper semi-continuity of global attractor when $\nu\rightarrow0^{+}$.

Freitas et al. \cite{Fr2022} studied nonlinear piezoelectric beam system with delay term:
\begin{equation*}
\begin{cases}
\rho v_{tt}-\alpha v_{xx}+\gamma\beta p_{xx}+f_{1}(v,p)+v_{t}=h_{1},\\[2mm]
\mu p_{tt}-\beta p_{xx}+\gamma\beta v_{xx}+f_{2}(v,p)+\mu_{1}p_{t}+\mu_{2}p_{t}(x,t-\tau)=h_{2},\\[2mm]
\end{cases}
\end{equation*}
and proved the system is asymptotically smooth gradient, and then they used stability estimation to obtain that  the system is quasi-stable. Finally, they also proved the global attractor has finite fractal dimension.

Ma et al. \cite{Ma2020} studied long-time dynamics of semilinear wave equation:
\begin{equation*}
\partial_{t}^{2} u-\Delta u+a(x) g\left(\partial_{t} u\right)+f(u)=\varepsilon h(x) ~\quad \text { in } \Omega \times \mathbb{R}^{+},
\end{equation*}
where $\Omega\subset\mathbb{R}^{3}$ is a bounded domain with a smooth boundary $\partial \Omega$ and $\varepsilon\in[0,1]$. In this paper, they mainly proved the existence and properties of global attractors and attractors are continuous under autonomous perturbations.

Freitas et al. \cite{Fre2021} investigated the following wave equations:
\begin{equation*}
  \begin{cases}
  u_{t t}-\Delta u+(-\Delta)^{\alpha_{1}} u_{t}+g_{1}\left(u_{t}\right)=f_{1}(u, v)+\varepsilon h_{1}, & \text { in } \Omega \times \mathbb{R}^{+}, \\
  v_{t t}-\Delta v+(-\Delta)^{\alpha_{2}} v_{t}+g_{2}\left(v_{t}\right)=f_{2}(u, v)+\varepsilon h_{2}, & \text { in } \Omega \times \mathbb{R}^{+},
  \end{cases}
\end{equation*}
where $\Omega\subset\mathbb{R}^{2}$ is a bounded domain with smooth boundary $\partial \Omega$, $\alpha_{1},\alpha_{2}\in(0,1)$ and $\varepsilon\in[0,1]$. The authors showed that  the dynamical system is quasi-stable and that the family of the global attractors is the continuous with respect to the parameter.

Recently, Freitas and Ramos \cite{Frei2021} also considered a system modeling a mixture of three interacting continua with external forces with a parameter $\varepsilon\in[0,1]$,
and gave the smoothness of global attractors, the continuity of global attractors.

The main contributions of this paper are:

(1) We consider the piezoelectric system with nonlinear damping and external forces with a parameter, we prove the existence of solutions by maximal monotone operator theory.

(2) We establish a quasi-stability estimate to prove the system is quasi-stable and then we obtain the existence of exponential attractor and some properties of the global attractor.

(3) We consider the upper semicontinuity of global attractors with respect to the parameter $\varepsilon\in[0,1]$.

The present paper is organized  as following. In section 2, we mainly establish the well-posedness of the solution by nonlinear operator theory. In section 3, we apply the infinite dynamical system theory to prove the existence of global attractors and exponential attractors. In section 4, we mainly show the upper semi-continuity of the global attractors.

\section{Existence of global solutions}
\setcounter{equation}{0}

In this section, we first give some assumptions and notations. Then, we are concerned with well-posedness of the solution of the system \eqref{1.1}-\eqref{1.2} by the similar argument as in  \cite{Ch2013,Pa1983,Pei2014} and the references therein.

\subsection{Assumption}
Now, we give  some assumptions which  will be used hereinafter. And the following $C,C_{i}$ all denote positive generic constants, which may be different in various
lines.\\
(1) The function $F\in C^{2}(\mathbb{R}^2)$ satisfies
 \begin{equation}\label{2.1}
\nabla F=(f_{1},f_{2}),
\end{equation}
and there exists $\beta_{0}, m_{F}\geq 0$ such that
 \begin{equation}\label{2.2}
 F(v,p)\geq -\beta_{0}(|v|^{2}+|p|^{2})-m_{F},
\end{equation}
where
 \begin{equation}\label{2.3}
 0\leq \beta_{0}< \frac {1}{2\beta_{1}},
\end{equation}
and $\beta_{1}$ is the embedding constant defined in \eqref{3.6}. Moreover, there exist constant $r\geq1,C_{f}>0$ such that
 \begin{equation}\label{2.4}
 |f_{i}(v)-f_{i}(p)|\leq C_{f}(1+|v|^{r-1}+|p|^{r-1})|v-p|.
 \end{equation}
Furthermore, for arbitrary $v,p\in \mathbb{R}$, we get
\begin{equation}\label{2.5}
\nabla F(v,p)\cdot (v,p)-F(v,p)\geq -\beta_{0}(|v|^{2}+|p|^{2})-m_{F}.
\end{equation}
(2) The external forces $h_{1},h_{2}\in L^{2}(0,L)$, $\varepsilon\in[0,1]$.\\
(3) For the nonlinear damping $g_{i}(i=1,2)$, we have
 \begin{equation}\label{2.6}
g_{i}\in C^{1}(\mathbb{R}),\quad g_{i}(0)=0,
\end{equation}
and there exists $m,M_{1},q\geq1$,
 \begin{equation}\label{2.7}
 m\leq g'_{i}(v)\leq M_{1}(1+|v|^{q-1}),~~~~~\quad \forall v\in\mathbb{R}.
\end{equation}
If $q\geq3$, there exist $l>q-1,M_{2}>0$ such that
 \begin{equation}\label{2.8}
g_{i}(v)v\geq M_{2}|v|^{l},~~~\quad|v|\geq1.
\end{equation}
Furthermore, from \eqref{2.7}, we obtain
 \begin{equation}\label{2.9}
(g_{i}(p)-g_{i}(v))(p-v)\geq m|p-v|^{2},~~~\quad p,v\in \mathbb{R}.
\end{equation}

\subsection{The Cauchy problem}

 In the rest of this paper, we denote
\begin{equation*}
 \|v\|_{r}=\|v\|_{L^{r}(0,L)},\quad r\geq1, \qquad (v,p)_{2}=(v,p)_{L^{2}(0,L)}.
\end{equation*}
We shall define the Sobolev Space
\[H^{i}_{*}(0,L)=\{v\in H^{i}(0,L):v(0)=0\},~~\quad i=1,2.\]
Because of $v(0)=0$, we obtain the Poincar\'{e}'s inequality
 \begin{equation}\label{2.10}
\lambda_{1}\|v\|^{2}_{2}\leq \|v_{x}\|^{2}_{2},~~~\quad \forall v\in H^{1}_{*}(0,L),
\end{equation}
where $\lambda_{1}>0$. Hence,  we deduce $\|v\|_{H^{1}_{*}(0,L)}=\|v_{x}\|_{2}$.

The space $\mathcal{H}$ is defined by
 \begin{equation*}
\mathcal{H}=V\times H=(H^{1}_{*}(0,L))^{2}\times (L^{2}(0,L))^{2}.
\end{equation*}
The inner product on $\mathcal{H}$ is
\begin{equation*}
  (z,\tilde{z})_{\mathcal{H}}=\alpha_{1}\int_{0}^{L}{v_{x}\tilde{v}_{x}}dx+\beta\int_{0}^{L}{(\gamma v_{x}-p_{x})(\gamma \tilde{v}_{x}-\tilde{p}_{x})}dx+\rho\int_{0}^{L}{v_{t}\tilde{v}_{t}}dx+\mu\int_{0}^{L}
  {p_{t}\tilde{p_{t}}}dx,
\end{equation*}
where $z=(v,p,v_{t},p_{t})^{T},\tilde{z}=(\tilde{v},\tilde{p},\tilde{v_{t}},\tilde{p_{t}})^{T}$.\\
From the inner product, we can define the norm as
\begin{equation}\label{2.11}
  \|z(t)\|^{2}_{\mathcal{H}}=\|(v,p,\tilde{v},\tilde{p})^{T}\|^{2}_{\mathcal{H}}=\alpha_{1}\|v_{x}\|^{2}+\beta\|\gamma v_{x}-p_{x}\|^{2}+\rho\|\tilde{v}\|^{2}+\mu\|\tilde{p}\|^{2},
\end{equation}
Moreover, there exists a constant $\kappa>0 $ such that
 \begin{equation}\label{2.12}
\|v\|_{2}^{2}+\|p\|_{2}^{2}\leq \frac{1}{\lambda_{1}}(\|v_{x}\|_{2}^{2}+\|p_{x}\|_{2}^{2})\leq \kappa(\alpha_{1}\|v_{x}\|^{2}_{2}+\beta\|\gamma v_{x}-p_{x}\|^{2}_{2}),
\end{equation}
In fact, observing  that
$$\|p_{x}\|_{2}^{2}=\|\gamma v_{x}-p_{x}-\gamma v_{x}\|^{2}_{2}\leq2\|\gamma v_{x}-p_{x}\|^{2}_{2}+2\gamma^{2}\|v_{x}\|_{2}^{2},$$
we have
$$\|v_{x}\|_{2}^{2}+\|p_{x}\|_{2}^{2}\leq(1+2\gamma^{2})\|v_{x}\|_{2}^{2}+2\|\gamma v_{x}-p_{x}\|^{2}_{2}.$$
where $\kappa=\max\{(2\gamma^{2}+1)\alpha^{-1}_{1},2\beta^{-1}\}$, the inequality \eqref{2.12} holds. Combining the Poincar\'{e}'s inequality \eqref{2.10} and the above formula, there exists $\beta_{1}=\frac{\kappa}{\lambda_{1}}>0$ such that
 \begin{equation}\label{2.13}
\|v\|_{2}^{2}+\|p\|_{2}^{2}\leq \beta_{1}(\alpha_{1}\|v_{x}\|^{2}_{2}+\beta\|\gamma v_{x}-p_{x}\|^{2}_{2}).
\end{equation}
\hspace*{0.7cm}Let us write the system \eqref{1.1}-\eqref{1.2} as an equivalent Cauchy problem
 \begin{equation}\label{2.14}
\frac{d}{dt}z(t)+\mathcal{A}z(t)=\mathcal{F}z(t),~~~\quad z(0)=z_{0},
\end{equation}
where $z_{0}=(v_{0},p_{0},v_{1},p_{1})^{T}\in \mathcal{H},$ and $ \mathcal{A}:D(\mathcal{A})=\{(v,p,\tilde{v},\tilde{p})^{T}\in \mathcal{H}: v,p\in H^{2}(0,L),\tilde{v},\tilde{p}\in H^{1}_{*}(0, L), v_{x}(L)=p_{x}(L)=0\}\subset\mathcal{H}\hookrightarrow\mathcal{H}$ is defined by
\begin{equation*}
  \mathcal{A}z(t)=
\begin{pmatrix}
-\tilde{v}\\
-\tilde{p}\\
\frac{1}{\rho}(-\alpha v_{xx}+\gamma\beta p_{xx}+g_{1}(v_{t}))\\
\frac{1}{\mu}(-\beta p_{xx}+\gamma\beta v_{xx}+g_{2}(p_{t}))
\end{pmatrix}.
\end{equation*}
The function $\mathcal{F}:\mathcal{H}\rightarrow\mathcal{H}$ is defined by
\begin{equation}\label{2.15}
  \mathcal{F}(z)=
\begin{pmatrix}
0\\
0\\
\frac{1}{\rho}(\varepsilon h_{1}-f_{1}(v,p)\\
\frac{1}{\mu}(\varepsilon h_{2}-f_{2}(v,p)
\end{pmatrix}.
\end{equation}
By a simple calculation, we have
 \begin{equation}\label{2.16}
(\mathcal{A}z,z)_{\mathcal{H}}=\int_{0}^{L}{(g_{1}({\tilde{v}})\tilde{v}+g_{2}({\tilde{p}})\tilde{p})}dx\geq0,
\end{equation}

\subsection{Energy Identities}

\begin{definition}
Let $z(t)=(v,p,v_{t},p_{t})^{T}\in C([0,\infty),\mathcal{H})$ be  a weak solution of the system \eqref{1.1}-\eqref{1.2},
if for any  $\varphi,\psi\in H_{*}^{1}(0,L)$ it satisfies
\begin{gather*}
 \rho \frac{d}{dt}(v_{t},\varphi)_{2}+\mu\frac{d}{dt}(p_{t},\psi)_{2}+\alpha_{1}(v_{x},p_{x})_{2}+\beta(\gamma v_{x}-p_{x})_{2}(\gamma\varphi_{x}-\psi_{x})_{2}+(g_{1}(v_{t}),\varphi)_{2}\\+(g_{2}(p_{t}),\psi)_{2}+\int_{0}^{L}
 {f_{1}(v,p)\varphi}dx+\int_{0}^{L}{f_{2}(v,p)\psi}dx
=(\varepsilon h_{1},\varphi)_{2}+(\varepsilon h_{2},\psi)_{2}.
\end{gather*}
Moreover, if
\[z\in C([0,\infty);D(\mathcal{A}))\cap C^{1}([0,\infty);\mathcal{H}),\]
then $z$ is called the strong solution.
\end{definition}

The total energy of solutions of \eqref{1.1}-\eqref{1.2} is  defined by
 \begin{equation}\label{2.17}
\mathcal{E}(t)=E(t)+\int_{0}^{L}{F(v(t),p(t))}dx-\varepsilon\int_{0}^{L}{(h_{1}v(t)+h_{2}p(t))}dx,~~~\quad t\geq0,
\end{equation}
where $E(t)=\frac{1}{2}\|z(t)\|_{\mathcal{H}}^{2}$.

\begin{lemma} If $z=(v,p,v_{t},p_{t})^{T}$ is a strong solution of \eqref{1.1}-\eqref{1.2}, the following conclusions will hold,\\
(1)\begin{equation}\label{2.18}
  \begin{aligned}
    \frac{d}{dt}\mathcal{E}(t)&=-\int(g_{1}(v_{t})v_{t}+g_{2}(p_{t})p_{t})dx \\
     &\leq -m(\|v_{t}(t)\|_{2}^{2}+\|p_{t}(t)\|_{2}^{2}).
  \end{aligned}
\end{equation}
(2) There exist constants $\beta_{2},C_{F}>0$ such that
\begin{equation}\label{2.19}
 \beta_{2}\|z\|_{\mathcal{H}}^{2}-C_{F}\leq \mathcal{E}\leq C_{F}(1+\|z\|_{\mathcal{H}}^{r+1}),~~\quad t\geq0.
\end{equation}
\end{lemma}

\begin{proof} Multiplying the equations in \eqref{1.1} by $v_{t}$, $p_{t}$, respectively, and using integration by parts, and applying the inequality \eqref{2.9}, we can obtain \eqref{2.18}.\\
Applying \eqref{2.2} and \eqref{2.14}, we have
\begin{equation}\label{2.20}
  \begin{aligned}
    \int_{0}^{L}{F(v,p)}dx &\geq -\beta_{0}(\|v\|^{2}_{2}+\|p\|^{2}_{2})-Lm_{F} \\
    & \geq -\beta_{0}\beta_{1}(\alpha_{1}\|v_{x}\|^{2}_{2}+\beta\|\gamma v_{x}-p_{x}\|^{2}_{2})-Lm_{F}\\
     &\geq -\beta_{0}\beta_{1}\|(v,p,v_{t},p_{t})^{T}\|^{2}_{\mathcal{H}}-Lm_{F}.
  \end{aligned}
\end{equation}
Using \eqref{2.3} and \eqref{2.20},
\begin{equation*}
  \begin{aligned}
    \mathcal{E}(t) &= \int_{0}^{L}{F(v,p)}dx-\varepsilon  \int_{0}^{L}{(h_{1}v+h_{2}p)}dx +E(t) \\
    & \geq (\frac{1}{2}-\beta_{0}\beta_{1})\|z\|_{\mathcal{H}}^{2}-Lm_{F}-\varepsilon  \int_{0}^{L}{(h_{1}v+h_{2}p)}dx.
\end{aligned}
\end{equation*}
Let
\begin{equation}\label{2.21}
\beta_{2}=\frac{1}{4}(1-2\beta_{0}\beta_{1})>0,
\end{equation}
since $\varepsilon\in[0,1]$, we have
\begin{equation}\label{2.22}
  \varepsilon  \int_{0}^{L}{(h_{1}v+h_{2}p)}dx\leq \frac{\beta_{2}}{\beta_{1}}(\|v\|^{2}_{2}+\|p\|^{2}_{2})+\frac{\beta_{1}}{4\beta_{2}}(\|h_{1}\|^{2}_{2}+\|h_{2}\|^{2}_{2}),
\end{equation}
the first part of \eqref{2.19} is obtained with
\begin{equation*}
  C_{F}=Lm _{F}+\frac{\beta_{1}}{4\beta_{2}}(\|h_{1}\|^{2}_{2}+\|h_{2}\|^{2}_{2}).
\end{equation*}
Moreover, by \eqref{2.4} we have
\begin{equation}\label{2.23}
  \int_{0}^{L}{F(v,p)}dx\leq C(\|v_{x}\|^{r+1}_{2}+\|p_{x}\|^{r+1}_{2}+1).
\end{equation}
By \eqref{2.13} and \eqref{2.25}, we can deduce the second inequality in \eqref{2.19}.
\end{proof}

\subsection{Well-Posedness}

In this part, we will use the nonlinear operator theory to prove the well-posedness of solutions.
\begin{definition} Let
$X$ be a reflexive Banach space, the  operator $A: X\rightarrow X'$ is called monotone  if it satisfies
\[\langle Az^{1}-Az^{2},z^{1}-z^{2}\rangle\geq0,  \quad \forall z^{1},z^{2}\in D(A),\]
 Furthermore,
 if $(A+I): X\rightarrow X'\ is\ onto$, then $A$ is maximal.
 \end{definition}

\begin{definition}
 Let $X$ be a reflexive Banach space, the operator $B: X\rightarrow X'$ is called hemicontinuous, if
\[{\lim_{\lambda\to0}}\langle B(u+\lambda v),w\rangle=\langle Bu,w\rangle,~~~\quad \forall u,v,w\in X.\]
\end{definition}

\begin{lemma}
The operator $\mathcal{F}: \mathcal{H}\rightarrow\mathcal{H}$ defined in \eqref{2.15} is locally Lipschitz continuous.
\end{lemma}

\begin{proof} Let the solutions $z^{1},z^{2}\in \mathcal{H}$ such that  $\|z^{1}\|_{\mathcal{H}},\|z^{2}\|_{\mathcal{H}}\leq \mathcal{R}$, where $\mathcal{R}>0$.\\Then, we can deduce
\begin{equation*}
  \|\mathcal{F}(z^{1})-\mathcal{F}(z^{2})\|^{2}_{\mathcal{H}}=\frac{1}{\rho}\int_{0}^{L}{|f_{1}(v^{1},p^{1})-f_{1}(v^{2},p^{2})|^{2}}dx
  +\frac{1}{\mu}\int_{0}^{L}{|f_{2}(v^{1},p^{1})-f_{2}(v^{2},p^{2})|^{2}}dx.
\end{equation*}
By \eqref{2.4},
\begin{equation}\label{2.24}
  |f_{i}(v^{1},p^{1})-f_{j}(v^{2},p^{2})|^{2}\leq C_{f}(|v^{1}|^{r-1}+|p^{1}|^{r-1}+|v^{2}|^{r-1}+|p^{2}|^{r-1}+1)^{2}(|v^{1}-v^{2}|^{2}+|p^{1}-p^{2}|^{2}).
\end{equation}
It follows from  \eqref{2.24} that there exists some constant $C_{\mathcal{R}}>0$ such that
\begin{equation*}
  \int_{0}^{L}{|f_{i}(v^{1},p^{1})-f_{j}(v^{2},p^{2})|^{2}}dx\leq C_{\mathcal{R}}\|z^{1}-z^{2}\|^{2}_{\mathcal{H}},~~\quad i=1,2.
\end{equation*}
which implies that
\begin{equation*}
  \|\mathcal{F}(z^{1})-\mathcal{F}(z^{2})\|^{2}_{\mathcal{H}}\leq C_{\mathcal{R}}\|z^{1}-z^{2}\|_{\mathcal{H}}.
\end{equation*}
\end{proof}
Now, we are in the position to give the existence of the solutions.
\begin{theorem}
Suppose the assumptions \eqref{2.1}-\eqref{2.7} hold. If $z_{0}\in \mathcal{H}$, the system \eqref{1.1}-\eqref{1.2}  has a unique weak solution $zt()$ satisfies  $z\in C([0,\infty);\mathcal{H})$, and it depends continuously on  the initial data $z_{0}$. In addition, if $z_{0}\in D(\mathcal{A})$, the weak solution is strong solution.
\end{theorem}

\begin{proof} Similar to \eqref{2.15} and \eqref{2.16}, using the monotonicity of function $g_{i}, i=1,2$, we know $\mathcal{A}$ is monotone. For purpose of obtain $\mathcal{A}$ is a maximal monotone operator, we need to show that
there exists $z\in D(\mathcal{A})$, for arbitrary $w\in \mathcal{H}$ such that
\begin{equation*}
  \mathcal{A}z+z=w.
\end{equation*}
In fact, in  the following, we decompose the operator $\mathcal{A}$ as
\[\mathcal{A}=\left(
                \begin{array}{cc}
                  0 & -I \\
                  B & G \\
                \end{array}
              \right),
\]
where $B: D(B)=\{(v,p)\in H_{*}^{2}(0,L)\times H_{*}^{2}(0,L)\}\subset V\hookrightarrow H$,
\[B(v,p)=\left(
           \begin{array}{c}
             \frac{1}{\rho}(-\alpha v_{xx}+\gamma\beta p_{xx}) \\
             \frac{1}{\mu}(-\beta p_{xx}+\gamma\beta v_{xx}) \\
           \end{array}
         \right),
\]
 $G: H\rightarrow H$ and
\[G(v_{t},p_{t})=\left(
                   \begin{array}{c}
                     \frac{1}{\rho}g_{1}(v_{t}) \\
                     \frac{1}{\mu}g_{2}(p_{t}) \\
                   \end{array}
                 \right).
\]
 Writing $z=(v,p,v_{t},p_{t})^{T}=(\tau,\eta)^{T}\in V\times H$, $w=(w_{1},w_{2})^{T}\in V\times H$, so
 \begin{equation*}
   \tau-\eta=w_{1}, ~~~~~~\qquad B(\tau)+G(\eta)+\eta=w_{2},
 \end{equation*}
 then, we can analyze that $\eta\in V$.
 Hence, we need to prove $K(\eta)=(B+I)\eta+G(\eta): V\rightarrow V'$ is onto. By Corollary 2.2 of \cite{Ba2010}, we need only to prove $K$ is maximal monotone and coercive. From \eqref{2.11} and the embedding $V\hookrightarrow H=H'\hookrightarrow V'$ such that
 \begin{equation*}
 \langle B\tau,\eta\rangle=(\tau,\eta)_{V},\quad \tau,\eta\in V.~~~\qquad
   \langle \tau,\eta\rangle=(\tau,\eta)_{H},\quad \tau\in H,\eta\in V.
 \end{equation*}
 \hspace*{0.7cm}Firstly, let $\eta^{1},\eta^{2}\in V$, then
 \begin{equation}\label{2.25}
   \begin{aligned}
   \langle(B+I)(\eta^{1}-\eta^{2}),\eta^{1}-\eta^{2}\rangle &=\langle B(\eta^{1}-\eta^{2}),\eta^{1}-\eta^{2}\rangle+\langle\eta^{1}-\eta^{2},\eta^{1}-\eta^{2}\rangle \\
  &=\|\eta^{1}-\eta^{2}\|_{V}^{2}+\|\eta^{1}-\eta^{2}\|_{H}^{2}\geq 0.
\end{aligned}
 \end{equation}
 By \eqref{2.9}, we have
 \begin{equation}\label{2.26}
   \langle G(\eta^{1})-G(\eta^{2}),\eta^{1}-\eta^{2}\rangle=(G(\eta^{1})-G(\eta^{2}),
   \eta^{1}-\eta^{2})_{H}\geq0.
 \end{equation}
So we obtain that $B+I$ and $G$ are monotone.\\
\hspace*{0.7cm}Secondly, let $\tau,\eta,\phi\in V $ where $\tau=(\tau^{1},\tau^{2},\tau^{3}),\eta=(\eta^{1},\eta^{2},\eta^{3}),
\phi=(\phi^{1},\phi^{2},\phi^{3})$, and $\lambda\in \mathbb{R},\lambda_{n}\rightarrow0$, we have
\begin{equation*}
  |\langle (B+I)(\tau+\lambda \eta),\phi\rangle-\langle(B+I)\tau,\phi\rangle|=|\lambda \langle(B+I)\eta,\phi\rangle|\leq |\lambda|(\langle By,\phi\rangle+\langle \eta,\phi\rangle),
\end{equation*}
which implies that the continuity of $\langle (B+I)(\tau+\lambda \eta),\phi\rangle$  at $\lambda=0$.
Then
\[\langle G(\tau+\lambda_{n}\eta),\phi\rangle=\sum_{i=1}^{3}\int_{0}^{L}{g_{i}(\tau^{i}+
\lambda_{n}\eta^{i})\phi^{i}}dx.\]
Clearly
\[g_{i}(\tau^{i}+\lambda_{n}\eta^{i})\phi^{i} \rightarrow g_{i}(\tau^{i})\phi^{i} \quad a.e\ in\ (0,L).\]
By \eqref{2.7} and Dominated Convergence Theorem, we deduce
\[\lim_{n\to\infty}\langle G(\tau+\lambda_{n}\eta),\phi\rangle=(G(\tau),\phi)_{H}.\]
So we obtain that $B+I$ and $G$ are hemicontinuous.

In addition, we can deduce $B+I$ and $G$ are coercive from \eqref{2.25} and \eqref{2.26}.

Therefore, according to Theorem 2.6 of \cite{Ba2010}, we know that $K$ is coercive and maximal monotone, which implies that  $K$ is onto. That is, $\mathcal{A}$ is maximal monotone in $\mathcal{H}$.

In conclusion, since $\mathcal{A}$ is maximal monotone and $\mathcal{F}$ is locally Lipschitz, by applying Theorem 7.2 of \cite{Ch2002}, we can obtain: when $z_{0}\in \mathcal{H}$, the problem \eqref{2.14} has a unique weak solution $z(t)\in C([0,t_{max});\mathcal{H})$. Moreover, if $t_{max}<\infty $, then $\lim\sup_{t\rightarrow t_{max}^{-}}\|z\|_{\mathcal{H}}=\infty$. When $z_{0}\in D(\mathcal{A})$, the problem \eqref{2.14} has a unique strong solution $z(t)\in C([0,t_{max}), D(\mathcal{A})),\ (t_{max}\leq\infty)$.

Now, we need to  show the existence of global solutions, that is $t_{\max}=\infty.$ In fact, let $z(t)$ be a strong solution defined in $[0, t_{\max})$. From \eqref{2.18}, we infer
\begin{equation}\label{2.27}
  \mathcal{E}(t)\leq\mathcal{E}(0).
\end{equation}
It follows from \eqref{2.19} and \eqref{2.17} that
\begin{equation*}
  \|z\|_{\mathcal{H}}^{2}\leq \frac{1}{\beta_{2}}(\mathcal{E}(0)+C_{F}).
\end{equation*}
By density argument, the conclusion for weak solution also holds. Therefore $t_{max}=\infty$.\\
Finally, let $z^{1},z^{2}$ be two weak solutions, by the standard arguments, for any $T>0$, there exists constant $C>0$ such that
\begin{equation}\label{2.28}
  \|z^{1}(t)-z^{2}(t)\|_{\mathcal{H}}^{2}\leq e^{CT}\|z^{1}(0)-z^{2}(0)\|_{\mathcal{H}}^{2}, \quad t\in [0,T].
\end{equation}
The proof is complete.
\end{proof}

\section{The existence of global attractor}
In this section, for  the sake of completeness,  we collect some known results in
 the theory of nonlinear dynamical systems (see \cite{Bu2008,Ch2010,Ro2001,Te1997}).

Let $(v,p,v_{t},p_{t})^{T}$ be the unique solution for the system \eqref{1.1}-\eqref{1.2}. We can define  the operator $S(t):\mathcal{H}\rightarrow \mathcal{H}$ by  $$S(t)(v_{0},p_{0},v_{1},p_{1})^{T}=(v(t),p(t),v_{t}(t),p_{t}(t))^{T},~~\quad t\geq0.$$ Hence, $(\mathcal{H},S(t))$ constitutes a dynamical system.

\begin{definition}
 Let $B\subset \mathcal{H}$ be a positively invariant set of a dynamical system $(\mathcal{H},S(t))$.

1. A function $\Phi(\mathfrak{y})$ is said to be a Lyapunov function, if $t\rightarrow \Phi(S(t)\mathfrak{y})$ is a non-increasing function for any $\mathfrak{y}\in B$.

2. A Lyapunov function is called strict, if there exist $t_{0}>0$, $\mathfrak{y}\in B$ such that $\Phi(S(t_{0})\mathfrak{y})=\Phi (\mathfrak{y})$, then $\mathfrak{y}=S(t)\mathfrak{y},\ t>0$.
\end{definition}
\begin{definition}
The dynamical system $(\mathcal{H},S_{t})$ is quasi-stable on $B\subset\mathcal{H}$, if there exist a compact seminorm $\chi_{X}(\cdot)$ on the space $X$ and nonnegative scalar function $\mathfrak{a}(t),\mathfrak{b}(t),\mathfrak{c}(t)\in\mathbb{R^{+}}$ satisfy:

(1) $\mathfrak{a}(t),\mathfrak{c}(t)$ are locally bounded on $[0,\infty)$;

(2) $\mathfrak{b}(t)\in L^{1}(\mathbb{R^{+}})$, and $\lim_{t\rightarrow\infty}\mathfrak{b}(t)=0$;

(3) for  any $ \mathfrak{y}_{1},\mathfrak{y}_{2}\in B$ and $t\geq0$, the estimates
\begin{equation*}
  \|S(t)\mathfrak{y}_{1}-S(t)\mathfrak{y}_{2}\|_{\mathcal{H}}^{2}\leq \mathfrak{a}(t)\|\mathfrak{y}_{1}-\mathfrak{y}_{2}\|_{\mathcal{H}}^{2},
\end{equation*}
and
\begin{equation*}
  \|S(t)\mathfrak{y}_{1}-S(t)\mathfrak{y}_{2}\|_{\mathcal{H}}^{2}\leq \mathfrak{b}(t)\|\mathfrak{y}_{1}-\mathfrak{y}_{2}\|_{\mathcal{H}}^{2}+
  \mathfrak{c}(t)\sup_{0\leq s\leq t}[\chi_{X}(\mathfrak{u}^{1}(s)-\mathfrak{u}^{2}(s))]^{2},
\end{equation*}
hold, where $S(t)=(\mathfrak{u}^{i}(t),\mathfrak{u}_{t}^{i}(t)), i=1,2$.
\end{definition}

\begin{lemma}
The dynamical system $(\mathcal{H},S(t))$ is gradient, that is, there exists a strict Lyapunov function $\Phi\in \mathcal{H}$. What's more,
$$\Phi(z)\rightarrow \infty \quad \Longleftrightarrow\quad \|z\|_{\mathcal{H}}\rightarrow \infty.$$
\end{lemma}

\begin{proof}
Let $\mathcal{E}(t)$ defined in \eqref{2.17} be a Lyapunov function $\Phi$ and $z_{0}=(v_{0},p_{0},v_{1},p_{1})^{T}\in\mathcal{H}$, we can infer that $t\rightarrow \Phi(S(t)z_{0})$ is a non-increasing function from \eqref{2.18}.\\
\hspace*{0.7cm}Supposed that $\Phi(S(t)z_{0})=\Phi(z_{0})$, for $t\geq0$. Then
$$\|v_t(t)\|_{2}^{2}=\|p_t(t)\|_{2}^{2}=0,~~\quad t\geq0.$$
We obtain
$$v_{t}=p_{t}=0,\quad a.e \ in\ (0,L),~~~\quad t\geq0.$$
Consequently, $v_{t}(t)=v_{0},p_{t}(t)=p_{0}$, which implies $S(t)z_{0}=z(t)=(v_{0},p_{0},0,0)^{T}$ is a stationary solution of $(\mathcal{H},S(t))$.

From \eqref{2.19}, we obtain
\begin{equation*}
  \Phi(z)\leq C_{F}(1+\|z\|_{\mathcal{H}}^{r+1}),~~\quad t\geq0.
\end{equation*}
Let $\Phi(z)\rightarrow\infty$, we have $\|z\|_{\mathcal{H}}\rightarrow\infty$. Then, it follows from \eqref{2.19} that
\begin{equation*}
  \|z\|_{\mathcal{H}}^{2}\leq \frac{\Phi(z)+C_{F}}{\beta_{2}}.
\end{equation*}
As a result, we infer from  $\|z\|_{\mathcal{H}}\rightarrow\infty$ that $\Phi(z)\rightarrow \infty $.
\end{proof}

\begin{lemma}
The set of stationary points $\mathcal{N}$ of $S(t)$ is bounded in $\mathcal{H}$.
\end{lemma}

\begin{proof} Let $z=(v,p,0,0)^{T}\in \mathcal{N}$ be the stationary solution of problem \eqref{1.1}-\eqref{1.2}. Then, we have the  following elliptic equations
\begin{equation}\label{3.1}
 \begin{cases}
-\alpha v_{xx}+\gamma\beta p_{xx}+f_{1}(v,p)=\varepsilon h_{1},\\
-\beta p_{xx}+\gamma\beta v_{xx}+f_{2}(v,p)=\varepsilon h_{2}.
\end{cases}
\end{equation}
Multiplying the equations in \eqref{3.1} by $v$, $p$, respectively, and integrating the result over $(0,L)$,  we have
\begin{equation*}
  \alpha_{1}\|v_{x}\|^{2}_{2}+\beta\|\gamma v_{x}-p_{x}\|^{2}_{2}=-
\int_{0}^{L}{(f_{1}(v,p)v+f_{2}(v,p)p)}dx+\varepsilon\int_{0}^{L}{(h_{1}v+h_{2}p)}dx.
\end{equation*}
Hence, using \eqref{2.1}, \eqref{2.2}, and \eqref{2.5}, we obtain
\begin{equation*}
  -\int_{0}^{L}{(f_{1}(v,p)v+f_{2}(v,p)p)}dx\leq 2\beta_{0}\beta_{1}(\alpha_{1}\|v_{x}\|^{2}_{2}+\beta\|\gamma v_{x}-p_{x}\|^{2}_{2})+2Lm_{F}.
\end{equation*}
By \eqref{2.21}, we have
\begin{equation*}
  4\beta_{2}(\alpha_{1}\|v_{x}\|^{2}_{2}+\beta\|\gamma v_{x}-p_{x}\|^{2}_{2})\leq 2Lm_{F}+\varepsilon\int_{0}^{L}{(h_{1}v+h_{2}p)}dx.
\end{equation*}
By Young's inequalities and \eqref{2.13}, we infer
\begin{equation*}
  \begin{aligned}
     \int_{0}^{L}{(h_{1}v+h_{2}p)}dx & \leq \frac{\beta_{2}}{\beta_{1}}(\|v\|_{2}^{2}+\|p\|_{2}^{2})+\frac{\beta_{1}}{4\beta_{2}} (\|h_{1}\|_{2}^{2}+\|h_{2}\|_{2}^{2})\\
      & \leq \beta_{2}(\alpha_{1}\|v_{x}\|^{2}_{2}+\beta\|\gamma v_{x}-p_{x}\|^{2}_{2})+\frac{\beta_{1}}{4\beta_{2}} (\|h_{1}\|_{2}^{2}+\|h_{2}\|_{2}^{2}).
   \end{aligned}
\end{equation*}
Therefore, we conclude
\begin{equation}\label{3.2}
  3\beta_{2}\|z\|^{2}_{\mathcal{H}} \leq 2Lm_{F}+\frac{\beta_{1}}{4\beta_{2}}(\|h_{1}\|^{2}_{2}+\|h_{2}\|^{2}_{2}).
\end{equation}
The proof is complete.
\end{proof}

\begin{lemma} Suppose the Assumption 2.1 holds. Let $B$ be a bounded forward invariant set in $\mathcal{H}$ and $S(t)z^{i}=(v^{i},p^{i},v_{t}^{i},p_{t}^{i})^{T}$ be a weak solution of the problem \eqref{1.1}-\eqref{1.2} with $z_{0}^{i}\in B, i=1,2$. Then there exist constant $\sigma,\varsigma,C_{B}>0$ independent of $\varepsilon$ such that
\begin{equation*}
  E(t)\leq \varsigma E(0)e^{-\sigma t}+C_{B}\sup_{s\in[0,t]}(\|v(s)\|^{2}_{2\theta}+\|p(s)\|^{2}_{2\theta}),\quad \forall t\geq0.
\end{equation*}
where $E(t)=\frac{1}{2}\|z\|_{\mathcal{H}}^{2}$, $v=v^{1}-v^{2}, p=p^{1}-p^{2},\theta\geq2$.
\end{lemma}

\begin{proof}
Let $F_{i}(v,p)=f_{i}(v^{1},p^{1})-f_{i}(v^{2},p^{2}), i=1,2, G_{1}(v_{t})=g_{1}(v_{t}^{1})-g_{1}(v_{t}^{2}),
G_{2}(p_{t})=g_{2}(p_{t}^{1})-g_{2}(p_{t}^{2})$.
Then $v=v^{1}-v^{2},p=p^{1}-p^{2}$ satisfy
\begin{equation}\label{3.3}
  \begin{cases}
\rho v_{tt}-\alpha v_{xx}+\gamma\beta p_{xx}+G_{1}(v_{t})=-F_{1}(v,p),\\[2mm]
\mu p_{tt}-\beta p_{xx}+\gamma\beta v_{xx}+G_{2}(v_{t})=-F_{2}(v,p),\\[2mm]
(v(0),p(0),v_{t}(0),p_{t}(0))=z^{1}-z^{2},\\[2mm]
v(0)=p(0)=v_{x}(L)=p_{x}(L)=0.
\end{cases}
\end{equation}
Multiplying the first equation of \eqref{3.3} by $v$, the second one by $p$, and integrating over $[0,L]\times [0,T]$,  we have
\begin{equation*}
  \begin{aligned}
    \int_{0}^{T}{E(t)}dt=&\int_{0}^{L}{(\rho v_{t}v+\mu p_{t}p)}dx\Big|_{0}^{T}+\int_{0}^{T}{(\rho\|v_{t}\|^{2}_{2}+\mu\|p_{t}\|^{2}_{2})}dt \\
    &-\frac{1}{2}\int_{0}^{T}\int_{0}^{L}{(G_{1}(v_{t})v)+G_{2}(p_{t})p)}dxdt-
  \frac{1}{2}\int_{0}^{T}\int_{0}^{L}{(F_{1}(v,p)v+F_{2}(v,p)p)}dxdt.
  \end{aligned}
\end{equation*}

 \textbf{Step 1.} By Poinc\'{a}re's and H\"{o}lder's inequalities, we have
\begin{equation*}
 \int_{0}^{L}{(\rho v_{t}v+\mu p_{t}p)}dx\leq {C}E(t).
\end{equation*}
Then
\begin{equation*}
  \int_{0}^{L}{(\rho v_{t}v+\mu p_{t}p)}dx\Big|_{0}^{T}\leq {C}(E(0)+E(T)),
\end{equation*}
where ${C}$ is a constant and ${C}>0$.
From \eqref{2.9}, we conclude
\begin{equation*}
  \int_{0}^{T}{(\rho\|v_{t}\|^{2}_{2}+\mu\|p_{t}\|^{2}_{2})}dt\leq \frac{1}{m}\int_{0}^{T}\int_{0}^{L}{(G_{1}(v_{t})v)+G_{2}(p_{t})p)}dxdt.
\end{equation*}

\textbf{Step 2.} According to Young's inequality and (2.9), we can deduce
\begin{equation*}
  \int_{0}^{T}\int_{0}^{L}{G_{1}(v_{t})v}dxdt\leq
  \frac{1}{2}\int_{0}^{T}\int_{0}^{L}{G_{1}(v_{t})v_{t}}dxdt+\frac{1}{2}\int_{0}^{T}\int_{0}^{L}
  {G_{1}(v_{t})\frac{|v|^{2}}{v_{t}}}dxdt.
\end{equation*}
Applying \eqref{2.7}, we infer
\begin{equation*}
  \int_{0}^{L}{G_{1}(v_{t})v}dx\leq \frac{1}{2}\int_{0}^{L}{G_{1}(v_{t})v_{t}}dx+
  \frac{M_{1}}{2}\int_{0}^{L}(1+|v_{t}^{1}|^{q-1}+|v_{t}^{2}|^{q-1})|v|^{2}dx.
\end{equation*}
For further estimation, we divide it  into three cases.

\textbf{Case a: $q=1$. } It is easy to get
\begin{equation*}
 \int_{0}^{L}(1+|v_{t}^{1}|^{q-1}+|v_{t}^{2}|^{q-1})|v|^{2}dx\leq
 3\|v\|_{2}^{2}(1+\int_{0}^{L}{(g_{1}(v_{t}^{1})v_{t}^{1}+g_{1}(v_{t}^{2})v_{t}^{2})}dx).
\end{equation*}

\textbf{Case b: $1<q<3$. }  Applying the assumption \eqref{2.9} and H\"{o}lder's inequality, we have
\begin{equation*}
  \begin{aligned}
    \int_{0}^{L}(1+|v_{t}^{1}|^{q-1}+|v_{t}^{2}|^{q-1})|v|^{2}dx & \leq {C}\|v\|_{2d_{1}}^{2}( \int_{0}^{L}{(1+|v_{t}^{1}|^{d(q-1)}+|v_{t}^{2}|^{d(q-1)})}dx)^{\frac{1}{d}}\\
     & \leq {C}\|v\|_{2d_{1}}^{2}(L+\int_{0}^{L}{(g_{1}(v_{t}^{1})v_{t}^{1}+g_{1}(v_{t}^{2})v_{t}^{2})}dx)^
     {\frac{1}{d}}\\
     & \leq {C}\|v\|_{2d_{1}}^{2}(L+\int_{0}^{L}{(g_{1}(v_{t}^{1})v_{t}^{1}+g_{1}(v_{t}^{2})v_{t}^{2})}dx),
  \end{aligned}
\end{equation*}
where $d=\frac{2}{q-1}$, and $\frac{1}{d}+\frac{1}{d_{1}}=1$.

\textbf{Case c: $q\geq3$.} By the similar argument as in  Case b, let $d=\frac{l}{q-1}>1$, we can have the result.\\
To sum up, for some ${C}>0,\theta\geq2$, we can infer
\begin{equation}\label{3.4}
  \int_{0}^{L}{G_{1}(v_{t})v}dx\leq \frac{1}{2}\int_{0}^{L}{G_{1}(v_{t})v_{t}}dx+ {C}\|v\|_{\theta}^{2}(L+\int_{0}^{L}{(g_{1}(v_{t}^{1})v_{t}^{1}+g_{1}(v_{t}^{2})v_{t}^{2})}dx).
\end{equation}
Similarly,
\begin{equation}\label{3.5}
  \int_{0}^{L}{G_{2}(p_{t})p}dx\leq \frac{1}{2}\int_{0}^{L}{G_{2}(p_{t})p_{t}}dx+ {C}\|p\|_{\theta}^{2}(L+\int_{0}^{L}{(g_{2}(p_{t}^{1})p_{t}^{1}+g_{2}(p_{t}^{2})p_{t}^{2})}dx).
\end{equation}
Furthermore, due to $z^{1},z^{2}\in B$, and by \eqref{2.18}, \eqref{2.19}, we conclude there exists ${C}_{B}>0$ such that
\begin{equation*}
  \int_{0}^{L}{(g_{1}(v_{t}^{1})v_{t}^{1}+g_{1}(v_{t}^{2})v_{t}^{2})}dx\leq {C}_{B},
\end{equation*}
\begin{equation*}
  \int_{0}^{L}{(g_{2}(p_{t}^{1})p_{t}^{1}+g_{2}(p_{t}^{2})p_{t}^{2})}dx\leq {C}_{B}.
\end{equation*}
Combining with \eqref{2.8}, \eqref{3.4}, \eqref{3.5}, and applying $L^{2\theta}(0,L)\hookrightarrow L^{\theta}(0,L)$,
we deduce
\begin{equation*}
 \begin{aligned}
   -\frac{1}{2}\int_{0}^{T}\int_{0}^{L}{(G_{1}(v_{t})v+G_{2}(p_{t})p)}dxdt & \leq \frac{1}{2}\int_{0}^{T}\int_{0}^{L}{(G_{1}(v_{t})v_{t}+G_{2}(p_{t})p_{t})}dxdt \\
    & +
  {C}_{T}\sup_{s\in[0,T]}(\|v(s)\|^{2}_{2\theta}+\|p(s)\|^{2}_{2\theta}).
 \end{aligned}
\end{equation*}

\textbf{Step 3.} Applying \eqref{2.4}, $H_{*}^{1}(\Omega)\hookrightarrow L^{r}(\Omega),r\in(1,\infty)$ and H\"{o}lder's inequality, we can infer
\begin{equation}\label{3.6}
  \begin{aligned}
    \int_{0}^{L}{F_{1}(v,p)v}dxdt&\leq C_{f}(\|v^{1}\|_{2\theta}^{r-1}+\|v^{2}\|_{2\theta}^{r-1}+\|p^{1}\|_{2\theta}^{r-1}+
    \|p^{2}\|_{2\theta}^{r-1})(\|v\|_{2\theta}+\|p\|_{2\theta})\|v\|_{2}\\
     & \leq {C}_{B}((\|v\|_{2\theta}+\|p\|_{2\theta})\|v\|_{2}\\
     & \leq {C}_{B}((\|v\|_{2\theta}^{2}+\|p\|_{2\theta}^{2}).
  \end{aligned}
\end{equation}
Similarly,  we obtain
\begin{equation}\label{3.7}
  \int_{0}^{L}{F_{2}(v,p)p}dxdt\leq {C}_{B}((\|v\|_{2\theta}^{2}+\|p\|_{2\theta}^{2}).
\end{equation}
Combining \eqref{3.7} with \eqref{3.6}, there exists ${C}_{T}>0$, we have
\begin{equation*}
  -\frac{1}{2}\int_{0}^{T}\int_{0}^{L}{(F_{1}(v,p)v+F_{2}(v,p)p)}dxdt\leq {C}_{T}\sup_{s\in[0,T]}(\|v(s)\|^{2}_{2\theta}+\|p(s)\|^{2}_{2\theta}).
\end{equation*}
Therefore, combining the above estimates, we have
\begin{equation}\label{3.8}
  \begin{aligned}
    \int_{0}^{T}{E(t)}dt & \leq {C}_{B}\int_{0}^{T}\int_{0}^{L}{(G_{1}(v_{t})v_{t}+G_{2}(p_{t})p_{t})}dxdt\\
    & +{C}_{T}\sup_{s\in[0,T]}(\|v(s)\|^{2}_{2\theta}+\|p(s)\|^{2}_{2\theta})+{C}(E(0)+E(T)).
  \end{aligned}
\end{equation}
for some constants  ${C}_{B},{C}_{T}>0$.\\
\textbf{Step 4.} Multiplying the equations of \eqref{3.3} by $v_{t},p_{t}$, respectively, and integrating over $[0,L]\times[s,T]$, we obtain
\begin{equation}\label{3.9}
  E(t)=E(s)-\int_{s}^{T}\int_{0}^{L}{(G_{1}(v_{t})v_{t}+G_{2}(p_{t})p_{t})}dxdt-
  \int_{s}^{T}\int_{0}^{L}{(F_{1}(v,p)v_{t}+F_{2}(v,p)p_{t})}dxdt.
\end{equation}
Due to
\begin{equation*}
 -\int_{s}^{T}\int_{0}^{L}{(G_{1}(v_{t})v_{t}+G_{2}(p_{t})p_{t})}dxdt\leq0,
\end{equation*}
we obtain
\begin{equation*}
  E(t)\leq E(s)-
  \int_{s}^{T}\int_{0}^{L}{(F_{1}(v,p)v_{t}+F_{2}(v,p)p_{t})}dxdt.
\end{equation*}
By the similar argument, we have
\begin{equation*}
 \begin{aligned}
   -\int_{0}^{L}{F_{1}(v,p)v_{t}}dxdt&\leq {C}_{B}((\|v\|_{2\theta}+\|p\|_{2\theta})\|v_{t}\|_{2}\\
    & \leq\zeta \|v_{t}\|_{2}^{2}+\frac{{C}_{B}}{4\zeta}(\|v\|_{2\theta}^{2}+\|p\|_{2\theta}^{2}).
 \end{aligned}
\end{equation*}
Analogously,
\begin{equation*}
 -\int_{0}^{L}{F_{2}(v,p)p_{t}}dxdt\leq\zeta \|p_{t}\|_{2}^{2}+\frac{{C}_{B}}{4\zeta}(\|v\|_{2\theta}^{2}+\|p\|_{2\theta}^{2}).
\end{equation*}
Consequently,
\begin{equation}\label{3.10}
  -\int_{0}^{L}{(F_{1}(v,p)v_{t}+F_{2}(v,p)p_{t})}dxdt\leq\zeta E(t)+\frac{{C}_{B}}{4\zeta}(\|v\|_{2\theta}^{2}+\|p\|_{2\theta}^{2}).
\end{equation}
Let $\zeta=\frac{1}{T}$, we have
\begin{equation*}
  E(T)\leq \frac{1}{T}\int_{0}^{T}{E(t)}dt+T{C}_{B}\int_{0}^{T}{(\|v\|_{2\theta}^{2}+
  \|p\|_{2\theta}^{2})}dt+E(s).
\end{equation*}
Then, integrating in $[0,T]$, there exists a constant ${C}_{T}>0$ such that
\begin{equation}\label{3.11}
  TE(T)\leq2\int_{0}^{T}{E(t)}dt+{C}_{T}\sup_{s\in[0,T]}(\|v(s)\|^{2}_{2\theta}+\|p(s)\|^{2}_{2\theta}).
\end{equation}

\textbf{Step 5.} Let \eqref{3.9} with $s=0$ and \eqref{3.10} with $\zeta=1$, we have
\begin{equation*}
  \int_{0}^{T}\int_{0}^{L}{(G_{1}(v_{t})v_{t}+G_{2}(p_{t})p_{t})}dxdt\leq E(0) -E
(T)+\int_{0}^{T}{E(t)}dt+T{C}_{B}\sup_{s\in[0,T]}(\|v(s)\|^{2}_{2\theta}+\|p(s)\|^{2}_{2\theta}).
\end{equation*}
 Combining the above estimates with \eqref{3.8}, we can deduce
 \begin{equation*}
  \int_{0}^{T}{E(t)}dt\leq {C}_{T}\sup_{s\in[0,T]}(\|v(s)\|^{2}_{2\theta}+\|p(s)\|^{2}_{2\theta})+{C}_{B}(E(0)+E(T)).
 \end{equation*}
 Then, from \eqref{3.11}
 \begin{equation*}
 TE(T)\leq {C}_{T}\sup_{s\in[0,T]}(\|v(s)\|^{2}_{2\theta}+\|p(s)\|^{2}_{2\theta})+{C}_{B}(E(0)+E(T)).
 \end{equation*}
 Let $T>2{C}_{B}$, for $m_{T}=\frac{{C}_{B}}{T-{C}_{B}}<1$, we can write
 \begin{equation*}
   E(T)\leq m_{T}E(0)+{C}_{T}\sup_{s\in[0,T]}(\|v(s)\|^{2}_{2\theta}+\|p(s)\|^{2}_{2\theta}).
 \end{equation*}
 Let $M_{n}=\sup_{s\in[nT,(n+1)T]}(\|v(s)\|^{2}_{2\theta}+\|p(s)\|^{2}_{2\theta}),
 n\in\mathbb{N}$. Then repeat the above argument progress on any $[nT,(n+1)T]$, we have
\begin{equation*}
  \begin{aligned}
    E(nT)&=m_{T}^{n}E(0)+{C}_{T}\sum^{n}_{k=1}{m_{T}^{n+1-l}M_{j-1}}\\
    & \leq m_{T}^{n}E(0)+\frac{{C}_{T}}{1-m_{T}}\sup_{s\in[0,mT]}(\|v(s)\|^{2}_{2\theta}+\|p(s)\|^{2}_{2\theta}).
  \end{aligned}
\end{equation*}
For any $t\geq0$, and $n\in\mathbb{N}, k\in[0,T)$ so it has $t=nT+k$, then
\begin{equation*}
  E(t)\leq E(nT)\leq m_{T}^{-1}m_{T}^{\frac{t}{T}}E(0)+\sup_{s\in[0,t]}(\|v(s)\|^{2}_{2\theta}+\|p(s)\|^{2}_{2\theta}).
\end{equation*}
Consequently, for $\sigma=-\frac{ln(m_{T})}{T}, \varsigma=m_{T}^{-1}, C_{B}=\frac{{C}_{T}}{1-m_{T}}$,
\begin{equation*}
  E(t)\leq \varsigma E(0)e^{-\sigma t}+C_{B}\sup_{s\in[0,t]}(\|v(s)\|^{2}_{2\theta}+\|p(s)\|^{2}_{2\theta}),\quad \forall t\geq0.
\end{equation*}
The proof is complete.
\end{proof}

\begin{lemma}
Let $B\subset\mathcal{H}$ be a bounded forward invariant set, then dynamical system $(\mathcal{H},S(t))$ is quasi-stable.
\end{lemma}
\begin{proof}
Defining $S(t)z^{i}=(v^{i}(t),p^{i}(t),v_{t}^{i}(t),p_{t}^{i}(t))^{T}$ for $z^{i}\in B,i=1,2$, and
$v=v^{1}-v^{2},p=p^{1}-p^{2}$. Then, it follows from \eqref{2.28} that
\begin{equation*}
  \|S(t)z^{1}-S(t)z^{2}\|_\mathcal{H}^{2}\leq a(t)\|z^{1}-z^{2}\|_\mathcal{H}^{2},
\end{equation*}
where $\mathfrak{a}(t)=e^{CT}$.\\
Now let $X=H_{*}^{1}(\Omega)\times H_{*}^{1}(\Omega)$ and the seminorm be
$$\chi_{X}(v,p)=(\|v\|_{2\theta}^{2}+\|p\|_{2\theta}^{2})^{\frac{1}{2}}.$$
Since $H_{*}^{1}(\Omega)\hookrightarrow\hookrightarrow L^{2\theta}(\Omega)$, we can obtain that  $\chi_{X}$ is compact on $X$.

According to Lemma 3.3, we have
$$\|S(t)z^{1}-S(t)z^{2}\|_\mathcal{H}^{2}\leq \mathfrak{b}(t)\|z^{1}-z^{2}\|_\mathcal{H}^{2}+\mathfrak{c}(t)
\sup_{s\in[0,t]}[\chi_{X}(v(s),p(s))]^{2},$$
with $\mathfrak{b}(t)=\varsigma e^{-\sigma t}$, $\mathfrak{c}(t)=C_{B}$. \\
It is easy to verify that  $\mathfrak{b}(t)\in L^{1}(\mathbb{R}^{+})$ and   $\lim_{t\rightarrow\infty}{\mathfrak{b}(t)}=0$.\\
Then, since $B$ is a bounded subset of $\mathcal{H}$, we have $\mathfrak{c}(t)$  is locally bounded on $[0,\infty)$. By the Definition 3.2, we have the dynamical system is quasistable on $B\subset\mathcal{H}$.
\end{proof}
Since dynamical system $(\mathcal{H},S(t))$ is quasi-stable, we can give our main results as following.
\begin{theorem} Under the assumptions of Theorem 2.1, we obtain

(1) The dynamical system $(\mathcal{H},S(t))$ has a global attractor $\mathcal{A} \subset \mathcal{H}$ which is compact and connected. Moreover, the attractor can be characterized by the unstable manifold $$\mathcal{A}=\mathbb{M}^{+}(\mathcal{N}),$$
emanating from the set of stationary solutions established in Lemma 3.2.

(2) The attractor $\mathcal{A}$ has finite fractal dimension $dim_{\mathcal{H}}\mathcal{A}$.

(3) Every trajectory stabilizes to the set $\mathcal{N}$, that is,
\[\lim_{t\rightarrow+\infty}dist_{\mathcal{H}}(S_{t}z,\mathcal{N})=0 \quad \text {for any} \ z\in \mathcal{H}.\]
In particular, there exists a global minimal attractor $\mathcal{A}_{min}$ to the dynamical system, which
is precisely characterized  by the set of the stationary points $\mathcal{N}$, that is $\mathcal{A}_{min}=\mathcal{N}$.

(4) The attractor is bounded in $\mathcal{H}_{1}=(H^{2}(0,L)\cap H_{*}^{1}(0,L))^{2}\times (H_{*}^{1}(0,L))^{2}$, and every trajectory $z=(v,p,v_{t},p_{t})^{T}$ in $\mathcal{A}$ has the property
\begin{equation}\label{3.12}
 \|(v,p)\|_{(H^{2}(0,L)\cap H_{*}^{1}(0,L))^{2}}^{2}+\|(v_{t},p_{t})\|_{(H_{*}^{1}(0,L))^{2}}^{2}+
\|(v_{tt},p_{tt})\|_{(L^{2}(0,L))^{2}}^{2}\leq R^{2},
\end{equation}
for some constant $R>0$ independent of $\varepsilon\in[0,1]$.
\end{theorem}

\begin{proof} (1) It follows from Lemma 3.4 that $(\mathcal{H},S(t))$ is quasi-stable.  Hence, we have that $(\mathcal{H},S(t))$ is asymptotically smooth by Proposition 7.9.4 of \cite{Ch2010}. Then, applying Lemma 3.1, Lemma 3.2 and Corollary 7.5.7 of \cite{Ch2010},
we can conclude $(\mathcal{H},S(t))$ possesses  a compact global  attractor $\mathcal{A}$. In addition, it can be characterized  by $\mathcal{A}=\mathbb{M}^{+}(\mathcal{N}).$

(2) Since  the system $(\mathcal{H},S(t))$ is quasi-stable, applying Theorem 7.9.6 of \cite{Ch2010}, we conclude that  attractor $\mathcal{A}$ has finite fractal dimension $dim_{\mathcal{H}}\mathcal{A}$.

(3) Combining Theorem 3.1-(1)  and Theorem 7.5.10 of \cite{Ch2010}, we can get the desired result immediately.

(4) Since $(\mathcal{H},S(t))$ is quasi-stable on $\mathcal{A}$, the arbitrary trajectory $z=(v,p,v_{t},p_{t})^{T}$ in $\mathcal{A}$ has the following regularity properties
\begin{align*}
  &v_{t},p_{t}\in L^{\infty}(\mathbb{R};H_{*}^{1}(0,L))\cap C(\mathbb{R};L^{2}(0,L)),\\
  & v_{tt},p_{tt}\in L^{\infty}(\mathbb{R};L^{2}(0,L)).
\end{align*}
It follows from \eqref{1.1} that
\begin{equation}\label{3.13}
  \begin{cases}
\alpha v_{xx}=\rho v_{tt}+\gamma\beta p_{xx}+f_{1}(v,p)+g_{1}(v_{t})-\varepsilon h_{1},\\[2mm]
\beta p_{xx}=\mu p_{tt}+\gamma\beta v_{xx}+f_{2}(v,p)+g_{2}(p_{t})-\varepsilon h_{2}.\\[2mm]
\end{cases}
\end{equation}
Then, we can deduce
\begin{equation*}
  \alpha_{1}v_{xx}=\rho v_{tt}+\gamma\mu p_{tt}+\gamma f_{2}(v,p)+\gamma g_{2}(p_{t})-\gamma\varepsilon h_{2}+f_{1}(v,p)+g_{1}(v_{t})-\varepsilon h_{1}.
\end{equation*}
Using the fact  $\alpha_{1}\neq0$, $v_{t},p_{t}\in L^{\infty}(\mathbb{R};H_{*}^{1}(0,L))\hookrightarrow L^{\infty}(\mathbb{R};L^{2}(0,L))$ and $f_{i}(v,p)$ is locally Lipschitz continuous, we have
\begin{equation*}
  v_{xx}\in L^{\infty}(\mathbb{R};L^{2}(0,L)),
\end{equation*}
\begin{equation*}
  v\in L^{\infty}(\mathbb{R};H^{2}(0,L)\cap H_{*}^{1}(0,L)).
\end{equation*}
It follows from \eqref{3.13} that
\begin{equation*}
  \beta p_{xx}\in L^{\infty}(\mathbb{R};L^{2}(0,L)),
\end{equation*}
\begin{equation*}
  p\in L^{\infty}(\mathbb{R};H^{2}(0,L)\cap H_{*}^{1}(0,L)).
\end{equation*}
The proof is complete.
\end{proof}

\begin{theorem}The system $(\mathcal{H},S(t))$ has a  generalized fractal
exponential attractor. More precisely, for any given $\xi\in(0, 1]$, there
exists a generalized exponential attractor  $\mathcal{A}_{exp, \xi}$ in the extended
space $\mathcal{H}_{-\xi}$ which is defined as the interpolation of
\begin{equation*}
   \mathcal{H}_{0}:=\mathcal{H},\quad \mathcal{H}_{-1}:=(L^{2}(0,L))^{2}\times ( H_{*}^{-1}(0,L))^{2}.
\end{equation*}
\end{theorem}

\begin{proof}
Let  us take $\mathcal{B}=\{z\in \mathcal{H}|\Phi(z)\leq R\}$ where $\Phi$ is the strict Lyapunov functional given in Lemma 3.1. Then we can derive that for  sufficiently large $R$ that $\mathcal{B}$ is a positively invariant  bounded absorbing set, which shows that the system is quasi-stable on the set $\mathcal{B}$.

Then for solution $z(t)=S(t)z_0$ with initial data  $z(0)\in \mathcal{B}$,  we can derive that, for  any $T>0$,
$$
\int_{0}^{T}\left\|z_{t}(s)\right\|_{\mathcal{H}_{-1}}^{2} \mathrm{~d} s \leq C_{\mathcal{B} T}^{2},
$$
which shows that
$$
\left\|S\left(t_{1}\right) z-S\left(t_{2}\right) z\right\|_{\mathcal{H}_{-1}} \leq \int_{t_{1}}^{t_{2}}\left\|z_{t}(s)\right\|_{\mathcal{H}_{-1}} d s \leq C_{\mathcal{B} T}\left|t_{1}-t_{2}\right|^{\frac{1}{2}}
$$
where $\mathcal{C}_{\mathcal{B}T}$ is a positive constant and  $ t_{1},t_{2}\in[0,T]$.
Hence we obtain  that for any initial data $z_0\in \mathcal{H}$ the map $t\mapsto S(t)z_0$
is $\frac{1}{2}$-H\"{o}lder continuous in the extended phase space $\mathcal{H}_{-1}$.  Therefore, it follows from
Theorem 7.9.9 in \cite{Ch2002} that the dynamical system $(\mathcal{H}, S(t))$ possesses a generalized fractal
exponential attractor with finite fractal dimension  in the extended space $\mathcal{H}_{-1}$.

 Furthermore, by the standard interpolation theorem,  we can obtain the
existence of exponential attractors in the extended
space $\mathcal{H}_{-\xi}$  with $\xi\in (0, 1)$. The proof is complete.
\end{proof}

\section{Upper semicontinuity of global attractor }

In this section, we denote the attractor obtained in Theorem 3.1 as the $\mathcal{A}_{\varepsilon}$. Then, we investigate the upper semicontinuity of the attractors $\mathcal{A}_\varepsilon$ as $\varepsilon\rightarrow\varepsilon_0$.
\begin{definition} \cite{Ma2020}
Let $\Lambda$ be a complete metric space and $S_{\sigma}(t)$ a family of semigroups on $\mathcal{H}$, where $\sigma\in\Lambda$. The global attractors $\mathcal{A}_{\sigma}$ is called upper semicontinuous on $\sigma_{0}\in\Lambda$ if
\[\lim_{\sigma\rightarrow\sigma_{0}}dist_{\mathcal{H}}(\mathcal{A}_{\sigma},\mathcal{A}_{\sigma_{0}})=0,\]
where $dist_{\mathcal{H}}\{\mathbb{A},\mathbb{B}\}=\sup_{a\in \mathbb{A}}\inf_{b\in \mathbb{B}}d(\mathbb{A},\mathbb{B})$ expresses the Hausdorff semi-distance in $X$. Similarly,  $\mathcal{A}_{\sigma}$ is lower semicontinuous on $\sigma_{0}\in\Lambda$ if
\[\lim_{\sigma\rightarrow\sigma_{0}}dist_{\mathcal{H}}(\mathcal{A}_{\sigma_{0}},\mathcal{A}_{\sigma})=0.\]
Then $\mathcal{A}_{\sigma}$ is continuous on $\sigma_{0}\in\Lambda$ if
\[\lim_{\sigma\rightarrow\sigma_{0}}d_{\mathcal{H}}(\mathcal{A}_{\sigma},\mathcal{A}_{\sigma_{0}})=0,\]
where $d_{\mathcal{H}}(\mathbb{A},\mathbb{B})=\max\{dist_{\mathcal{H}}(\mathbb{A},\mathbb{B}),dist_{\mathcal{H}}(\mathbb{B},\mathbb{A})\}$ expresses the Hausdorff metric in $X$.
\end{definition}
\begin{proposition} \cite{Ho2015} Assume that\\
(H1) $(\mathcal{H},S_{\sigma}(t))$ has a global attractor $\mathcal{A}_{\sigma}$ for any $\sigma\in\Lambda$,\\
(H2) There exists a bounded set $B\subset\mathcal{H}$ such that $\mathcal{A}_{\sigma}\in B$ for every $\sigma\in\Lambda$,\\
(H3) $S_{\sigma}(t)z$ is continuous in $\sigma$ for $t>0$ and uniformly for $z$ in bounded subsets of $\mathcal{H}$.\\
Then the global attractor is continuous on all $I$, where $I$ is a "residual" set dense in $\Lambda$.
\end{proposition}
\begin{lemma}
There exists a set $I$ dense in $[0, 1]$ such that the global attractor $\mathcal{A}_{\varepsilon}$ obtained in  Theorem 3.1 is continuous at  $\varepsilon_{0}\in I$, that is
\begin{equation}\label{4.1}
    \lim_{\varepsilon\rightarrow\varepsilon_{0}}d_{\mathcal{H}}
   (\mathcal{A}_{\varepsilon},\mathcal{A}_{\varepsilon_{0}})=0.
\end{equation}
\end{lemma}

\begin{proof} The argument is inspired by \cite{Babin1997,Hale1989}. We apply Proposition 4.1 with $\Lambda=[0,1]$. Then Theorem 3.1 implies that the  assumption (H1) holds.

 It follows from  \eqref{2.19} and the fact $\sup_{z\in\mathcal{A}_{\varepsilon}}{\Phi_{\varepsilon}(z)}\leq
 \sup_{z\in\mathcal{N}_{\varepsilon}}\Phi_{\varepsilon}(z)$ that
 \begin{equation*}
   \begin{aligned}
      \sup_{z\in\mathcal{A}_{\varepsilon}}{\|z\|_{\mathcal{H}}^{2}} & \leq\frac{\sup_{z\in\mathcal{A}_{\varepsilon}}{\Phi_{\varepsilon}(z)+C_{F}}}{\beta_{2}} \\
      & \leq\frac{C_{F}\sup_{z\in\mathcal{N}_{\varepsilon}}\|z\|_{\mathcal{H}}^{r+1}+2C_{F}}{\beta_{2}}.
    \end{aligned}
 \end{equation*}
Hence, we can derive from  \eqref{3.2} that there exists a positive constant $C$ independent of $\varepsilon$ such that
 \begin{equation*}
    \sup_{z\in\mathcal{A}_{\varepsilon}}{\|z\|_{\mathcal{H}}^{2}}\leq C,~~\quad \forall\varepsilon\in[0,1].
 \end{equation*}
 Then we have that
$\mathcal{B}=\{z\in\mathcal{H}:\|z\|_{\mathcal{H}}^{2}\leq C\}$
is a bounded set which is independent of $\varepsilon$ and $\mathcal{A}_{\varepsilon}\subset \mathcal{B}$ for any $\varepsilon\in [0,1]$.
Therefore, the assumption (H2) holds.

Let $B$ be a bounded set of $\mathcal{H}$. Then for any given $\varepsilon_{1},\varepsilon_{2}\in [0,1], z_{0}\in B$, we define
\[S_{\varepsilon_{i}}(t)z_{0}=(v^{i}(t),p^{i}(t),v_{t}^{i}(t),p_{t}^{i}(t))^{T}, \quad i=1,2,\]
and
\[v=v^{1}-v^{2}, p=p^{1}-p^{2}.\]
Then $(v,p,v_t,p_t)$ satisfies the following equations
\begin{equation}\label{4.2}
 \begin{cases}
 \ \rho v_{tt}-\alpha v_{xx}+\gamma\beta p_{xx}=(\varepsilon_{1}-\varepsilon_{2})h_{1}-F_{1}(v,p)-G_{1}(v_{t}),\\[2mm]
\mu p_{tt}-\beta p_{xx}+\gamma\beta v_{xx}=(\varepsilon_{1}-\varepsilon_{2})h_{2}-F_{2}(v,p)-G_{2}(p_{t}),\\[2mm]
\end{cases}
\end{equation}
where $F_i(v,p)$ ($i=1,2$) $G_1(v_t)$ and $G_2(p_t)$ are constructed by the same as in Lemma 3.3.
Multiplying the equations \eqref{4.2} by $v_{t},p_{t}$, respectively. and integrating over $[0,L]$ by parts, we have
\begin{equation}\label{4.3}
\begin{aligned}
  \frac{1}{2} \frac{d}{dt}\|z\|_{\mathcal{H}}^{2} =&-\int_{0}^{L}{(F_{1}(v,p)+F_{2}(v,p))}dx-
   \int_{0}^{L}{(G_{1}(v_{t})v_t+G_{2}(p_{t}))p_t}dx \\
  & +(\varepsilon_{1}-\varepsilon_{2})\int_{0}^{L}{(h_{1}v_{t}+h_{2}p_{t})}dx.
\end{aligned}
\end{equation}
Using \eqref{2.4}, H\"{o}lder's inequality and $H_{*}^{1}(0,L)\hookrightarrow L^{r}(0,L),r\in[0,\infty)$, we can derive that
\begin{equation}\label{4.4}
\begin{aligned}
    \int_{0}^{L}{F_{1}(v,p)v_{t}}dx& \leq C_{f}(1+\|S_{\varepsilon_{1}}(t)z_{0}\|_{\mathcal{H}}^{r-1}+
    \|S_{\varepsilon_{2}}(t)z_{0}\|_{\mathcal{H}}^{r-1})(\|v\|_{2r}+\|p\|_{2r})
    \|v_{t}\|_{2}\\
   & \leq C_{f}(1+\|S_{\varepsilon_{1}}(t)z_{0}\|_{\mathcal{H}}^{r-1}+
    \|S_{\varepsilon_{2}}(t)z_{0}\|_{\mathcal{H}}^{r-1})(\|\nabla v\|_{2}+\|\nabla p\|_{2})
    \|v_{t}\|_{2}.
\end{aligned}
\end{equation}
Since  $\mathcal{E}_{\varepsilon}(t)$ is a non-increasing function, then for any $z_{0}\in B$, we have
\begin{equation*}
  \|S_{\varepsilon_{i}}(t)z_{0}\|_{\mathcal{H}}^{r-1}\leq\frac{\mathcal{E}_{\varepsilon_{i}}(0)+C_{F}}{\beta_{2}}
  \leq C_{B}, \,\,i=1,2.
\end{equation*}
Combining the above estimate with \eqref{4.4}, we have
\begin{equation*}
   \int_{0}^{L}{F_{1}(v,p)v_{t}}dx\leq C_{B}(\|\nabla v\|_{2}+\|\nabla p\|_{2})\|v_{t}\|_{2}\leq
    C_{B}(\|\nabla v\|_{2}^{2}+\|\nabla p\|_{2}^{2})+\|v_{t}\|_{2}^{2}.
\end{equation*}
Similarly,
\begin{equation*}
  \int_{0}^{L}{F_{2}(v,p)p_{t}}dx\leq C_{B}(\|\nabla v\|_{2}^{2}+\|\nabla p\|_{2}^{2})+\|p_{t}\|_{2}^{2}.
\end{equation*}
Therefore,
\begin{equation}\label{4.5}
  \int_{0}^{L}{(F_{1}(v,p)v_{t}+F_{2}(v,p)p_{t})}dx\leq C_{B}\|z\|_{\mathcal{H}}^{2}.
\end{equation}
It follows from  the monotonicity property \eqref{2.9} that
\begin{equation}\label{4.6}
  - \int_{0}^{L}{(G_{1}(v_{t})v_{t}+G_{2}(p_{t})p_{t})}dx\leq0.
\end{equation}
Moreover, it is easy to verify that
\begin{equation}\label{4.7}
  \begin{aligned}
    (\varepsilon_{1}-\varepsilon_{2})\int_{0}^{L}{(h_{1}v_{t}+h_{2}p_{2})}dx &\leq\frac{1}{2} (\|v_{t}\|_{2}^{2}+\|p_{t}\|_{2}^{2})+\frac{1}{2}|\varepsilon_{1}-\varepsilon_{2}|^{2}
    (\|h_{1}\|_{2}^{2}+\|h_{2}\|_{2}^{2})\\
     & \leq \frac{1}{2} \|z\|_{\mathcal{H}}^{2}+\frac{1}{2}|\varepsilon_{1}-\varepsilon_{2}|^{2}
    (\|h_{1}\|_{2}^{2}+\|h_{2}\|_{2}^{2}).
  \end{aligned}
\end{equation}
Combining \eqref{4.5}, \eqref{4.6}, \eqref{4.7} with  \eqref{4.3}, we have
\begin{equation}\label{4.8}
  \frac{d}{dt}\|z\|_{\mathcal{H}}^{2}\leq C_{B}\|z\|_{\mathcal{H}}^{2}+|\varepsilon_{1}-\varepsilon_{2}|^{2}
    (\|h_{1}\|_{2}^{2}+\|h_{2}\|_{2}^{2}).
\end{equation}
Using the Gronwall's inequality  and the fact $\|z(0)\|_{\mathcal{H}}^{2}=0$, we can derive from \eqref{4.8} that
\begin{equation*}
  \|z(t)\|_{\mathcal{H}}^{2}\leq\frac{1}{C_{B}}(e^{C_{B}t}-1)(\|h_{1}\|_{2}^{2}+\|h_{2}\|_{2}^{2})
  |\varepsilon_{1}-\varepsilon_{2}|,~~\quad t>0.
\end{equation*}
Hence we have
\begin{equation*}
  \|S_{\varepsilon_{1}}(t)z_{0}-S_{\varepsilon_{2}}(t)z_{0}\|_{\mathcal{H}}\leq
  (\frac{1}{C_{B}}(e^{C_{B}t}-1)(\|h_{1}\|_{2}^{2}+\|h_{2}\|_{2}^{2}))^{\frac{1}{2}} |\varepsilon_{1}-\varepsilon_{2}|.
\end{equation*}
So the assumption (H3) holds.
As a conclusion, the equality \eqref{4.1} holds by Proposition 4.1.
\end{proof}

\begin{theorem}
 Suppose the assumptions of Theorem 3.1 hold. Then the family of global attractors $\mathcal{A}_{\varepsilon}$ is upper semicontinuous at $\varepsilon_0$, namely,
\begin{equation}\label{4.9}
   \lim_{\varepsilon\rightarrow\varepsilon_{0}}dist_{\mathcal{H}}
   (\mathcal{A}_{\varepsilon},\mathcal{A}_{\varepsilon_{0}})=0.
\end{equation}
\end{theorem}

\begin{proof}
The argument is inspired by \cite{Ge2013,Ha1988}. Firstly, we suppose that \eqref{4.9} does not hold. Then there exist $\lambda>0$, the sequence $\varepsilon_{n}\rightarrow\varepsilon_{0}$ and $z_{0}^{n}\in \mathcal{A}_{\varepsilon^{n}}$ such that
\begin{equation}\label{4.10}
  dist_{\mathcal{H}}
   (z_{0}^{n},\mathcal{A}_{\varepsilon_{0}})\geq\lambda>0,~~\quad n\in\mathbb{N}.
\end{equation}

Let $z^{n}(t)=(v^{n}(t),p^{n}(t),v_{t}^{n}(t),p_{t}^{n}(t))^{T}$ be a bounded full trajectory from the attracator $\mathcal{A}_{\varepsilon^{n}}$ with  $z^{n}(0)=z_{0}^{n}$.
It follows from \eqref{3.12} that
${z^{n}}$ is bounded in $L^{\infty}(\mathbb{R};\mathcal{H}_{1})$.

Because of $\mathcal{H}_{1}\hookrightarrow\hookrightarrow\mathcal{H}$, applying Simon's Compactness Theorem (see \cite{Si1987}), we can get $z\in C([-T,T];\mathcal{H})$ and a subsequence $\{z^{n}\}$ such that,
$$z^{n}\rightarrow z\quad \text {in}\ C([-T,T];\mathcal{H}).$$
Then, we can conclude that
\begin{equation*}
  \sup_{t\in\mathbb{R}}{\|z(t)\|_{\mathcal{H}}}<\infty.
\end{equation*}
\hspace*{0.7cm}Let
$z(t)=(v(t),p(t),v_{t}(t),p_{t}(t))^{T}$ be a bounded full trajectory of the limiting semi-flow. Then, we can infer that $z$ solves the limiting equation ($\varepsilon=\varepsilon_{0}$), namely,
\begin{equation}\label{4.11}
\begin{cases}
\ \rho v_{tt}-\alpha v_{xx}+\gamma\beta p_{xx}+f_{1}(v,p)+g_{1}(v_{t})=\varepsilon_{0} h_{1},\\[2mm]
\mu p_{tt}-\beta p_{xx}+\gamma\beta v_{xx}+f_{2}(v,p)+g_{2}(p_{t})=\varepsilon_{0} h_{2}.\\[2mm]
\end{cases}
\end{equation}
In fact, from \eqref{1.1}, we can get $z^{n}$ satisfies
\begin{equation}\label{4.12}
\begin{cases}
\ \rho v^{n}_{tt}-\alpha v^{n}_{xx}+\gamma\beta p^{n}_{xx}+f_{1}(v^{n},p^{n})+g_{1}(v^{n}_{t})=\varepsilon^{n} h_{1},\\[2mm]
\mu p^{n}_{tt}-\beta p^{n}_{xx}+\gamma\beta v^{n}_{xx}+f_{2}(v^{n},p^{n})+g_{2}(p^{n}_{t})=\varepsilon^{n} h_{2}.\\[2mm]
\end{cases}
\end{equation}
we can use the same argument as in the proof of (H3) in Lemma 4.1, so we can infer that \eqref{4.11} is the limit of \eqref{4.12} as $n\rightarrow\infty$.\\
As a result, we have
\begin{equation*}
z_{0}^{n} \rightarrow z(0) \in \mathcal{A}_{\varepsilon_{0}},
\end{equation*}
 which contradicts \eqref{4.10}. The proof is complete.
\end{proof}

\section*{Acknowledgments}The authors would like to thank the referees for the careful reading of this
paper. This project is supported by NSFC (No. 11801145 and No. 12101189), the
Innovative Funds Plan of Henan University of Technology 2020ZKCJ09 and  the Fund of Young Backbone Teacher in Henan Province (No.2018GGJS068).


\begin{thebibliography}{99}


\bibitem{Ba2010} Barbu V., Nonlinear Differential Equations of Monotone Types in Banach Spaces, Springer,  New York, 2010.
\bibitem{Babin1997} Babin A.V., Pilyugin S.Yu., Continuous dependence of attractors on the shape of domain, J. Math. Sci. 87(2), 3304-3310 (1997).
\bibitem{Bu2008} Bucci F., Chueshov I., Long-time dynamics of a coupled system of nonlinear wave and thermoelastic plate equations, Discrete. Cont. Dyns. 22(3), 557-586 (2008).
\bibitem{Bu2014} Buxi D., Redout\'{e} J-M., Frequency sensing of medical signals using low-voltage piezoelectric sensors, Sens. Actuat. A. 220, 373-381 (2014).

\bibitem{Ch2002} Chueshov I., Eller M., On the attractor for a semilinear wave equation with critical exponent and nonlinear boundary dissipation, Comm. Partial Differ. Equ. 27(9), 1901-1951 (2002).

\bibitem{Ch2010} Chueshov I., Lasiecka I., Von Karman Evolution Equations: Well-posedness and Long Time Dynamics, New York, 2010.
\bibitem{Ch2013} Charles W., Soriano J.A., Decay rates for Bresse system with arbitrary nonlinear localized damping, J. Differ. Equ. 255, 2267-2290 (2013).

\bibitem{Da1978} Dakto R., Representation of solutions and stability of linear differential-difference equations in a Banach space, J. Differ. Equ. 29(1), 105-166 (1978).

\bibitem{Da1986} Dakto R., Lagnese J., An example on the effect of time delays in boundary feedback stabilization of wave equations, SIAM. J. Control. Optim. 24(1), 152-156 (1986).

\bibitem{Da1988} Dakto R., Not all feedback stabilized hyperbolic systems are robust with respect to small time delays in their feedbacks, SIAM. J. Cintrol. Optim. 26(3), 697-713 (1988).

\bibitem{Da1991} Datko, R., Two questions concerning the boundary control of certain elastic systems, J. Differ. Equ. 92(1), 27-44 (1991).

\bibitem{Da2016} Dagdeviren C., Recent progress in flexible and stretchable piezoelectric devices for mechanical energy harvesting, sensing and actuation, Extreme. Mech. Lett. 9(1), 269-281 (2016).
\bibitem{Dag2014} Dagdeviren C., Yang B.D., Conformal piezoelectric energy harvesting and storage from motions of the heart, lung, and diaphragm, Proc. Natl. Acad. Sci. U. S. A. 111(5), 1927-1932 (2014).

\bibitem{Fr2022} Freitas M.M., Ramos A.J.A., Dynamics of piezoelectric beams with magnetic effects and delay term, Evol. Equ. Control. Theory. 11(2), 583-603 (2022).

\bibitem{Fr2021} Freitas M.M., Ramos A.J.A., Long-time dynamics for a fractional piezoelectric system with magnetic effects and Fourier's law, J. Differ. Equ. 280, 891-927 (2021).
\bibitem{Fre2021} Freitas M.M., Dos Santos M.J., Quasi-stability and continuity of attractors for nonlinear system of wave equations, Nonauton. Dyn. Syst. 8(1), 27-45 (2021).
\bibitem{Frei2021} Freitas M.M., Ramos A.J.A., Existence and continuity of global attractors for ternary mixtures of solids, Discrete. Cont. Dyn-B. (2021).
\bibitem{Ge2013} Geredeli P.G., Lasiecka I., Asymptotic analysis and upper semicontinuity with respect to rotational inertia of attractors to
von Karman plates with geometrically localized dissipation and critical nonlinearity, Nonlinear. Anal. 91, 72-92 (2013).

\bibitem{Ha1989} Haraux A., Une remarque sur la stabilisation de certains syst\`{e}mes du deuxi\`{e}me ordre en temps, Port. Math. 46, 245-258 (1989).
\bibitem{Ha1988} Hale J.K., Raugel G., Upper semicontinuity of the attractor for a singulary perturbed hyperbolic equation, J. Differ. Equ. 73(2), 197-214 (1988).
\bibitem{Hale1989}  Hale J.K., Raugel G., Lower semicontinuity of attractors
of gradient systems and applications, Ann. Mat. Pura Appl. 154(1), 281-326 (1989).
\bibitem{Ho2015} Hoang L., Olson E., On the continuity of global attractors, P. Am. Math. Soc. 143(10), 4389-4395 (2015).

  \bibitem{Li2015} Liang Z.W., Li Y.M., Structure Optimization of a Grain Impact Piezoelectric Sensor and Its Application for Monitoring Separation Losses on Tangential-Axial Combine Harvesters, Sensors, 15(1), 1496-1517 (2015).

\bibitem{Ma1982} Maxwell J.C., A Dynamical Theory of the Electromagnetic Field, Scottish Academic Press, Edinburgh, 1982.
\bibitem{Ma1998} Maxwell J.C., A Treatise on Electricity and Magnetism, The Clarendon Press, Oxford University Press, New York, 1998.

\bibitem{Ma2015} Ma H.K., Luo W.F., Development of a piezoelectric micropump with novel separable design
for medical applications, Sens. Actuator A Phys. 236, 57-66 (2015).

\bibitem{Ma2020} Ma T.F., Seminario-Huertas P.N., Attractors for semilinear wave equations with localized damping and external forces,
Commun. Pur. Appl. Anal. 19(4), 2219-2233 (2020).

\bibitem{Mo2013} Morris K., \"{O}zer A.\"{O}., Strong Stabilization of Piezoelectric Beams with Magnetic Effects, The Proceedings of 52nd IEEE Conference on Decision and Control, Italy, 3014-3019 (2013).

\bibitem{Mo2014} Morris K.A., \"{O}zer A.\"{O}., Modeling and Stabilizability of Voltage-Actuated Piezoelectric Beams with Magnetic Effects, SIAM. J. Control. Optim. 52(4), 2371-2398 (2014).


\bibitem{z2015} \"{O}zer A.\"{O}., Further stabilization and exact observability results for voltage-actuated piezoelectric beams with magnetic effects, Math. Control. Signal. 27(2), 219-244 (2015).

\bibitem{z2021} \"{O}zer A.\"{O}., Stabilization Results for Well-Posed Potential Formulations of a Current-Controlled Piezoelectric Beam and Their Approximations, Appl Math Optim. 84, 877-914 (2021).

\bibitem{Pa1983} Pazy A., Semigroups of Linear Operators and Applications to PDE, Springer-Verlag, New York, 1983.


\bibitem{Pei2014}Pei P., Rammaha M.A., Local and global well-posedness of semilinear Reissner-Mindlin-Timoshenko plate equations, Nonlinear Anal. 105, 62-85 (2014).

\bibitem{Ra2018} Ramos A.J.A., Goncalves C.S.L., Exponential stability and numerical treatment for piezoelectric beams with magnetic effect, ESAIM-Math. Model. Numer. Anal. 53(1), 255-274 (2018).
\bibitem{Ra2019} Ramos A.J.A., Freitas M.M., Equivalence between exponential stabilization and boundary observability for piezoelectric beams with magnetic effect, Z. Angew. Math. Phys. 70(2), 1-14 (2019).
\bibitem{Ra2021} Ramos A.J.A.,  \"{O}zer A.\"{O}., Freitas M.M., Exponential stabilization of fully dynamic and electrostatic piezoelectric beams with delayed distributed damping feedback, Z. Angew. Math. Phys. 72(26), 1-15 (2021).
\bibitem{Ro2001} Robinson J.C., Infinite-Dimensional Dynamical Systems. An introduction to dissipative parabolic PDEs and the theory of global attractor, Cambridge University Press, 2001.

\bibitem{Sm2005} Smith R.C., Smart Material Systems: Model Development, SIAM, Philadelphia, 2005.


\bibitem{Si1987} Simon J., Compact sets in the space $L_{p}(0,T;B)$, Ann. Mat. Pura Appl. 146(4), 65-96 (1987).

\bibitem{Te1997} Temam R., Infinite-Dimensional Dynamical Systems in Mechanics and Physics, Springer-Verlag, New York, 1997.
\bibitem{Te2006} Tebou L.Y., Zuazua E., Uniform boundary stabilization of the finite difference space discretization of the 1-d wave equation, Adv. Comput. Math. 26(1), 337-365 (2006).


\bibitem{Xu2019} Xu L.Z., Wei C.C., Development of rapeseed cleaning loss monitoring system and experiments in a combine harvester, Biosyst. Eng. 178(2), 118-130 (2019).









\bibitem{Zh2015} Zhang S.Q., Li Y.X. Active shape and vibration control for piezoelectric bonded composite structures using various geometric nonlinearities, Compos. Struct. 122, 239-249 (2015).







\end{thebibliography}

\end{document}